\documentclass[12pt]{amsart}

\usepackage{hyperref}
\usepackage[dvips]{graphicx}

\newtheorem{theorem}{Theorem}[section]

\theoremstyle{definition}
\newtheorem{definition}[theorem]{Definition}

\theoremstyle{remark}

\newcommand{\CC}{{\mathbb C}}

\newcommand{\ff}{\varphi}

\newcommand{\alg}{\mathrm{alg}}

\newcommand{\cA}{\mathcal{A}}

\newcommand{\vN}{\text{vN}}

\title{Bi-free Probability Theory and Reflection Positivity}

\author{Roland Speicher}
\address{Saarland University, Department of Mathematics, D-66123 Saarbruecken, Germany}
\email{speicher@math.uni-sb.de}
\thanks{This work is dedicated to KR Parthasarathy whose pursuit of grasping the gist of the matter in the most systematic way served me always as inspiration and shining example.\\
I thank Michael Hartz, Gandalf Lechner, Michael Skeide and Yoshimichi Ueda for their remarks on previous versions of this work.}

\begin{document}

\maketitle

\begin{abstract}
We point out that bi-free product constructions respect reflection positivity.
\end{abstract}
\section{Introduction}

The goal in constructive quantum field theory is to contruct in a mathematical rigorous way non-trivial quantum field theories; either in the form of operator-valued distributions (Wightman functions, or Schwinger functions) or, alternatively via nets of operator-algebras. A crucial ingredient in such theories, which was isolated in the 1970's by Osterwalder and Schrader is the notion of 
\emph{reflection positivity}. However, up to now it is still an open problem to
find a non-trivial theory in 3+1-dimensional space-time which satisfies reflection positivity. 

One possible route to find such mathematical models is to construct more complex ones from simpler building blocks. One of the few universal constructions, which is at the heart of free probability theory, is the free product. However, this does not fit well with the requirements of quantum field theories. The physical axiom of locality requires that fields at space-like separated positions do not interact and hence should commute. This means that physical theories of quantum fields should have a large collection of operators which commute. Free products do not go well with commutativity. Even if we start with
algebras with many commuting operators, taking free products of those will usually
kill all commutativity. This problematic nature of free products with respect to commutativity has the effect that up to now free probability has not really played a vital role in the quest for such physical models. 

However, there have been some attempts in this direction, see \cite{ALR,LTU}, where the idea is to take free product nets, but then also look on the commutants of those objects. (For the relevance of commutants in algebraic quantum field theory, see, e.g., Section 5 of \cite{Gandalf}.) It seems to me that actually a recent extension of free probability, so-called bi-free probability theory \cite{Voi-bi1}, should give a more conceptual frame for dealing with such issues. Bi-free probability
allows, and in particular also preserves, some commutativity. Whereas free probability deals with objects like free group von Neumann algebras, bi-free probability takes those and their commutants at the same time.
It seems that the bi-free construction should allow to built out of simple models more complicated ones by taking bi-free products without destroying all commutativity relations. 

As a special instance of this possibility we want to consider the fate of reflection positivity under such a bi-free product. The setting of reflection positivity fits very well into the bi-free frame, since in both theories we have two
subalgebras which are like mirror images of each other (left and right faces in the bi-free setting, and positive and negative cones in the reflection positivity setting). Since the bi-free product respects the left-right distinction it is quite plausible that it should also respect reflection positivity. We will show that this is indeed the case. It should be remarked that this is nothing deep, but just a simple observation about the structure of the bi-free product -- the main concern is to point out that there is actually such a construction and thus giving some hope that (bi-)free probability tools might find their way into the construction of quantum fields via this route. 

\section{Bi-Freeness and Reflection Positivity}
We will here use the definition of reflection positivity from \cite{Jaffe}; for simplicity we restrict to the bosonic case. Furthermore, we will also ignore all analytical issues and work only on the level of algebras and non-commutative probability spaces. Extensions to operator-algebras and other analytic settings should follow by the usual continuity arguments. All our algebras are unital.
\begin{definition}
Let $\cA$ be an algebra, and $\cA^+,\cA^-$ subalgebras of $\cA$ such that
\begin{itemize}
\item
$\cA=\text{span}\{a_-a_+\mid a_-\in\cA^-, a_+\in\cA^+\} $
\item
$\cA^-$ and $\cA^+$ commute, i.e.
$$a_-a_+=a_+a_- \qquad\text{for all $a_-\in\cA_-$ and $a_+\in\cA_+$}$$
\end{itemize}
Furthermore, assume that $\theta:\cA\to\cA$ is an anti-linear homomorphism which squares to the identity, such that we have $\theta(\cA^+)=\cA^-$.
We say then that a linear functional $\tau:\cA\to\CC$ is \emph{reflection positive} (with respect to $(\cA,\cA^+,\cA^-,\theta)$) if we have 
$$\tau\bigl(\theta(a)a\bigr)\geq 0\qquad\text{for all $a\in\cA^+$}.$$
\end{definition}

The following is the adaptation of Voiculescu's definition of bi-freeness to our special situation. In the bi-free theory one usually has two distinct types of variables (or subalgebras), which are addressed as ``left'' and ``right'' variables. This corresponds clearly to $\cA^-$ and $\cA^+$. However, in the general bi-free setting left and right do not have to commute. If they do, as in our case, the system is called ``bipartite''. In this sense, the following definition is the special case of a bi-free product for a bipartite system.

\begin{definition}
Let, for each $i\in I$, the data $(\cA_i,\cA_i^+,\cA_i^-,\theta_i,\tau_i)$ be given. 
Then we construct their \emph{bi-free product} 
$$(\cA,\cA^+,\cA^-,\theta,\tau)=\ast_{i\in I}(\cA_i,\cA_i^+,\cA_i^-,\theta_i,\tau_i)$$ 
as follows. 
$\cA^+$ is the free product of the $\cA^+_i$; $\cA^-$ is the free product of the $\cA^-_i$; $\cA$ is the tensor product of $\cA^+$ and $\cA^-$; $\theta$ on $\cA^+$ or $\cA^-$ is defined as the free product of the $\theta_i$. The main point is the definition of $\tau$ such that the faces $(\cA^+,\cA^-)$
are becoming bi-free. Since we are in the bi-partite situation (i.e., all left variables commute with all right variables) we only have to specify 
$\tau(\theta(a)b)$ for $a,b\in \cA^+$.
Since $\cA^+$ is the free product of the $\cA^+_i$ we can write any element $b$
in $\cA^+$ as a linear combination of terms of the form
$b_n\cdots b_1$ for $n\geq 0$ and with $b_k\in \cA_{i(k)}^+$ such that
$i(k)\not=i(k+1)$ for all $k=1,\dots,n-1$ and such that $\tau_{i(k)}(b_k)=0$.
(For $n=0$ this shall mean that $b$ is a multiple of 1.)
Let $a_m\cdots a_1$ be another such term with $a_k\in\cA_{j(k)}^+$. Then
according to Lemma 2.2 from \cite{Voi-bi2} we put in such a situation
\begin{equation}\label{eq:def}
\tau(\theta(a_m)\cdots \theta(a_1) b_n\cdots b_1)=\delta_{mn}
\prod_{1\leq k\leq n} \delta_{i(k)j(k)} \tau(\theta(a_k)b_k).
\end{equation}
\end{definition}

Note that the above definition contains also that the $\cA_i^+$ are free with respect to $\tau$ (for $m=0$) and that the $\cA_i^-$ are free with respect to $\tau$ (for $n=0$). The relation between left and right variables is that $\cA_i^+$ is independent from $\cA_j^+$ for $j\not=i$, but the relation between $\cA_i^+$ and $\cA_i^-$ is determined by $\tau_i$. Hence we are free to specify the $\tau_i$ arbitrarily for each $i$ and the above bi-free product construction will then embedd these into a bigger space with producing commutation between left and right, but freeness both on the left and on the right side. The main observation which I want to make here is that actually reflection positivity is preserved under this construction.

\begin{theorem}
Assume that, for each $i\in I$, $\tau_i$ is reflection positive with respect to the data
$(\cA_i,\cA_i^+,\cA_i^-,\theta_i)$. Then the bi-free product state $\tau$ in
$$(\cA,\cA^+,\cA^-,\theta,\tau)=\ast_{i\in I}(\cA_i,\cA_i^+,\cA_i^-,\theta_i,\tau_i)$$ 
is also reflection positive.
\end{theorem}

\begin{proof}
This follows in the same way as the proof that the free product preserves positivity, see Theorem 6.13 in \cite{NSp}.
Namely, we can write any $a\in\cA^+$ as a linear combination
$$a=\sum_{n=0}^N\sum_{i_1,\dots,i_n\in I \atop i_1\not=i_2\not=\cdots \not=i_n} a_{i_1,\dots,i_n}$$
for some $N\geq 0$ and where $a_{i_1,\dots,i_n}$  is of the form $a_n\cdots a_1$ with $a_k\in \cA_{i_k}^+$ and $\tau_{i_k}(a_k)=0$. Then, by \eqref{eq:def}, we have that
$\tau(\theta(a_{i_1,\dots,i_n}) a_{j_1,\dots,j_m})=0$ unless $n=m$ and
$i_1=j_1$, \dots, $i_n=j_n$. So we have to prove positivity only for fixed $i_1,\dots, i_n$. Consider such an element and write it as
$$a_{i_1,\dots,i_n}=\sum_{k=1}^p a_n^{(k)}\cdots a_1^{(k)}$$
where $a_m^{(k)}\in \cA_{i_m}^+$ and $\tau_{i_m}(a_m^{(k)})=0$ for all $m=1,\dots,n$ and $k=1,\dots,p$.
Then we have
\begin{align*}
\tau(\theta(a_{i_1,\dots,i_n})\cdot a_{i_1,\dots,i_n})&=\sum_{k,l=1}^p
\tau\left( \theta(a_n^{(k)}\cdots a_1^{(k)})\cdot a_n^{(l)}\cdots a_1^{(l)}\right)\\
&=\sum_{k,l=1}^p \tau_{i_1}\left(\theta_{i_1}(a_1^{(k)})\cdot a_1^{(l)}\right)\cdots \tau_{i_n}\left(\theta_{i_n}(a_n^{(k)})\cdot a_n^{(l)}\right) 
\end{align*}
Now for each $m=1,\dots,n$ the matrix
$$\left(\tau_{i_m}\left( \theta_{i_m}(a_m^{(k)})\cdot a_m^{(l)}\right)\right)_{k,l=1}^p$$
is positive by our assumption that $\tau_{i_m}$ is reflection positive (note that $\theta_{i_m}$ is anti-linear).  Hence also the pointwise (Schur) product of all those $n$ matrices is positive; which gives then the positivity of 
$\tau(\theta(a_{i_1,\dots,i_n})\cdot a_{i_1,\dots,i_n})$. For details, see \cite{NSp}.
\end{proof}


\begin{thebibliography}{10}

\bibitem{ALR}
C. D'Antoni, R. Longo, and F. Radulescu: Conformal nets, maximal temperature and models from free probability. 
J. Operator Theory 45 (2001), 195--208.

\bibitem{Gandalf}
R. Correa da Silva and G. Lechner: Modular Structure and Inclusions of Twisted Araki-Woods Algebras. 
Commun. Math. Phys. 402 (2023): 2339-2386.

\bibitem{Jaffe}
A. Jaffe and B. Janssens: Reflection positive doubles.
J. Funct. Anal. 272 (2017), 3506--3557.

\bibitem{LTU}
R. Longo, Y. Tanimoto, and Y. Ueda: Free products in AQFT.
Annales de l'Institut Fourier 69 (2019), 1229--1258.


\bibitem{NSp}
A. Nica and R. Speicher: Lectures on the Combinatorics of Free Probability Theory. Cambridge University Press, 2006.

\bibitem{Voi-bi1}
D.-V. Voiculescu: Free Probability for Pairs of Faces I.
Commun. Math. Phys. 332 (2014), 955--980.

\bibitem{Voi-bi2}
D.-V. Voiculescu: Free Probability for Pairs of Faces II: 2-Variables Bi-free Partial $R$-Transform and Systems with Rank $\leq 1$ Commutation.
Annales de l'Institut Henri Poincar\'e, Probabilit\'ees et Statistiques, Vol. 52, No. 1, 2016, 1--15.

\end{thebibliography}
\end{document}